\documentclass[12pt,a4paper]{amsart} 
\usepackage{inputenc}
\usepackage{latexsym}
\usepackage{pinlabel}          
\usepackage[a4paper , lmargin = {2cm} , rmargin = {2cm} , tmargin = {2.5cm} , bmargin = {2.5cm} ]{geometry}             
\usepackage{graphicx}
\usepackage{amssymb, amsthm}
\usepackage{epstopdf}
\usepackage{romannum}
\usepackage{tikz}
\usepackage{subfig}
\usepackage[arrow, matrix, curve]{xy} 
\usepackage{hyperref}
\usetikzlibrary{arrows}
\usepackage{enumerate}  
 \usepackage{graphicx}
 \usepackage{caption}
 \usepackage{bm}

\theoremstyle{plain}

\newtheorem{theorem}{Theorem}[section]
\newtheorem{lemma}[theorem]{Lemma}

\newtheorem{corollary}[theorem]{Corollary}

\newtheorem{examples}[theorem]{Examples}
\newtheorem{proposition}[theorem]{Proposition}
\theoremstyle{definition}

\newtheorem{definition}[theorem]{Definition}

\title[]{The growth series of Dyer groups}

\author{Luis Paris and Olga Varghese}
\date{\today}

\address{Luis Paris\\
IMB, UMR 5584, CNRS, Universit\'{e} de Bourgogne, 21000 Dijon (France)}
\email{lparis@u-bourgogne.fr}

\address{Olga Varghese\\ Institute of Mathematics, Heinrich-Heine-University Düsseldorf, Universitätsstra{\upshape{\ss}}e 1, 40225, Düsseldorf, Germany.}
\email{olga.varghese@hhu.de}

\begin{document}
	
\pagenumbering{arabic}
	
\begin{abstract}
Graph products of cyclic groups and Coxeter groups are two families of groups that are defined by labeled graphs. The family of Dyer groups contains these both families and gives us a framework to study these groups in a unified way. This paper focuses on the growth series of a Dyer group $D$ with respect to the standard generating set. We give a recursive formula for the growth series of $D$ in terms of the growth series of standard parabolic subgroups. As an application we obtain the rationality of the growth series of a Dyer group. Furthermore, we show that the growth series of $D$ is closely related to the Euler characteristic of $D$.

\vspace{1cm}
\hspace{-0.6cm}
\medskip

\medskip
\hspace{-0.5cm}{\bf 2010 Mathematics Subject Classification.} 20F36, 20F55.

\thanks{The first author is supported by the French project “AlMaRe” (ANR-19-CE40-0001-01) of the ANR.
The second author is supported by DFG grant VA 1397/2-2.}

\end{abstract}
\maketitle

\section{Introduction}

\noindent
Let $(G,S)$ be a pair where $G$ is a group and $S=\left\{s_1, \ldots, s_k\right\}$ is a generating set of $G$. One way to study the group $G$ is by counting its elements algebraically/geometrically. 
Each element $g\in G$ can be written as a word $g=x_1\ldots x_n$ where each letter $x_i$, $i=1,\ldots, n$ lies in the alphabet $S\cup S^{-1}=\left\{s_1,\ldots, s_k, s_1^{-1},\ldots, s_k^{-1}\right\}$. The \emph{length of $g$}, denoted by $l(g)=l_S(g)$, is the minimal length  of a word expression of $g$ in the alphabet $S\cup S ^{-1}$. We count the number of elements of length $n$ in $G$ and convert this
sequence into a formal power series: $$\mathcal{G}_{(G,S)}(t):=\sum_{n=0}^\infty |\left\{g\in G\mid l(g)=n\right\}|\cdot t^n\,.$$
Thus, $\mathcal{G}_{(G,S)}(t)=\sum_{n=0}^\infty a_n\cdot t^n$ where $a_n$ is the number of vertices in a sphere of radius $n$ in the Cayley-graph $Cay(G,S)$. 
This formal power series is called \emph{the (standard) growth series of $G$ (with respect to $S$)} and tends to be an important measure of complexity for infinite groups. 
 
Let us calculate the growth series of the infinite cyclic group with the canonical generating set $\mathcal{G}_{(\mathbb{Z},\left\{1\right\})}(t)=1+2t+2t^2+2t^3+\ldots= 1+2(t+t^2+t^3\ldots)=1+\frac{2t}{1-t}=\frac{1+t}{1-t}$. Thus, $\mathcal{G}_{(\mathbb{Z},\left\{1\right\})}(t)$ is a rational function. By definition, a pair $(G,S)$ has \emph{rational growth series} if there exist  polynomials $f(t)$ and 
$g(t)$ with integer coefficients such that $\mathcal{G}_{(G,S)}(t)=\frac{f(t)}{g(t)}$. Many groups that appear in geometric group theory have rational growth series, for example Coxeter groups \cite{Solomon,Bourbaki}, surface groups \cite{Cannon,CannonWagreich}, virtually abelian groups \cite{Benson} and hyperbolic groups\cite{Gromov,GhysHarpe}. However, there exist nilpotent groups and finite generating sets, such that the growth series are not rational \cite{Stoll}.   
The central object under our investigation is the growth series of Dyer groups. 
Throughout this paper $(\Gamma, m, f)$ is a Dyer graph and $(D,V)$ where $V=V(\Gamma)$ is the associated Dyer system. This means that the vertex set $V(\Gamma)$ is finite and is endowed with a map $f\colon V(\Gamma)\to\mathbb{N}_{\geq 2}\cup\left\{\infty\right\}$ and the edge set $E(\Gamma)$ is endowed with a map $m\colon E(\Gamma)\to\mathbb{N}_{\geq 2}$. For two letters $a, b$ and a natural number $m$ we define $\pi(a,b,m):= abababa\ldots$ where the length of the word is $m$. 
Further, we assume that for every edge $e=\left\{x,y\right\}\in E(\Gamma)$ if $m(e)\neq 2$, then $f(x)=f(y)=2$. The associated \emph{Dyer group} is defined as follows
\begin{gather*}
D:=\langle V\mid x^{f(x)} \text{ if }f(x)\neq\infty,\ \pi(x,y,m(\left\{x,y\right\}))=\pi(y,x,m(\left\{x,y\right\}))\\
\text{ if }\left\{x,y\right\}\in E(\Gamma)\rangle\,.
\end{gather*}
We note that, if $f(x)=2$ for all $x\in V(\Gamma)$, then $D$ is a Coxeter group, and if $f(x)=\infty$ for all $x\in V(\Gamma)$, then $D$ is a right-angled Artin group. 

For a subset $Y\subseteq V$ we denote by $D_Y$ the subgroup in $D$ which is generated by the set $Y$. 
This subgroup is called  a \emph{standard parabolic subgroup}.
It is shown in \cite{Dyer} that $(D_Y,Y)$ is itself a Dyer system which is associated to the Dyer graph ($\Gamma_Y,m_Y,f_Y)$, where $\Gamma_Y$ is the full subgraph of $\Gamma$ spanned by $Y$, $m_Y$ is the restriction of $m$ to $E(\Gamma_Y)$, and $f_Y$ is the restriction of $f$ to $V(\Gamma_Y)=Y$.

Let $(D, V)$ be a Dyer system. We define $V_2:=\left\{x\in V\mid  f(x)=2\right\}$, $V_\infty:=\left\{x\in V\mid f(x)=\infty\right\}$, $V_p:=\left\{x\in V| 2<f(x)<\infty \right\}$.
Let $D_2$ resp. $D_\infty$ resp. $D_p$ be the subgroup of $D$ generated by $V_2$ resp. $V_\infty$ resp. $V_p$. The Dyer group $D$ is of \emph{spherical type} if $\Gamma$ is a complete graph and $D_2$ is finite.
In particular, if $D$ is of  spherical type, then $D=D_2\times D_p\times D_\infty$, $D_p$ is a finite abelian group, and $D_\infty=\mathbb{Z}^l$ where $l=|V_\infty|$.

\bigskip
We state now our main result.

\begin{theorem}
\label{MainTheorem}
Let $(D,V)$ be a Dyer system. 
\begin{enumerate}
\item If $D$ is not of spherical type, then
$$\frac{(-1)^{|V|+1}}{\mathcal{G}_{(D,V)}(t)}=\sum_{Y\subsetneq V}\frac{(-1)^{|Y|}}{\mathcal{G}_{(D_Y,Y)}(t)}\,.$$
\item If $D$ is of spherical type, then we decompose $D_p=\prod_{x\in V_p}\mathbb{Z}/f(x)\mathbb{Z}$ and let $l=|V_\infty|$. Then
$$\mathcal{G}_{(D,V)}(t)=\mathcal{G}_{(D_2,V_2)}(t)\cdot \mathcal{G}_{(D_p,V_p)}(t)\cdot \mathcal{G}_{(D_\infty,V_\infty)}(t)\,.$$
where
\begin{itemize}
\item $D_2$ is a finite Coxeter group, hence $\mathcal{G}_{(D_2,V_2)}(t)$ can be calculated using the formula for finite Coxeter groups \cite{Solomon}:
$$\mathcal{G}_{(D_2,V_2)}(t)=\prod_{i=1}^k(1+t+\ldots+t^{m_i})\,,$$
where $m_1,\dots,m_k$ are the exponents of $(D_2,V_2)$.
\item $\mathcal{G}_{(D_p,V_p)}(t)=\prod_{x\in V_p}\mathcal{G}_{(\mathbb{Z}/f(x)\mathbb{Z},\left\{1\right\})}(t)$. If $f(x)=2r$, then $$\mathcal{G}_{(\mathbb{Z}/f(x)\mathbb{Z},\left\{1\right\})}(t)=1+2t+2t^2+\ldots+2t^{r-1}+t^r\,.$$
If $f(x)=2r+1$, then $$\mathcal{G}_{(\mathbb{Z}/f(x)\mathbb{Z},\left\{1\right\})}(t)=1+2t+2t^2+\ldots+2t^{r}\,.$$
\item $\mathcal{G}_{(D_\infty,V_\infty)}(t)=\frac{(1+t)^l}{(1-t)^l}$.
\end{itemize}
\end{enumerate}
\end{theorem}

As a direct consequence we obtain the rationality of the growth series of a Dyer system.

\begin{corollary}
Let $(D,V)$ be a Dyer system. The growth series of $D$ with respect to $V$ is rational.
\end{corollary}

Often there are interesting connections between special values of the growth series of $(G,S)$ with other properties of a group $G$, for example it was proven in \cite{Serre} that for a Coxeter system $(W,S)$, the value $\mathcal{G}_{(W,S)}(1)$ is closely related to the rational Euler characteristic of $W$ which we denote by $\chi(W)$. More precisely:
$$\frac{1}{\mathcal{G}_{(W,S)}(1)}=\chi(W).$$
We prove that the same relation holds for all Dyer groups.

\begin{theorem}(see Theorem 5.3)
Let $(D,V)$ be a Dyer system. Then
$$\frac{1}{\mathcal{G}_{(D,V)}(1)}=\chi(D)\,.$$
\end{theorem}

\section{Preliminaries}

We start this chapter by reviewing some  standard facts of the word length and length functions.

\begin{definition}
Let $G$ be a finitely generated group and $S$ be a finite generating set.
\begin{enumerate}
\item For $g\in G$, $g\neq 1$  the \emph{word length of $g$} is defined as
$$l(g)=l_S(g)=min\left\{n\mid g=s_1^{\epsilon_1}s_2^{\epsilon_2}\ldots s_n^{\epsilon_n}, s_i\in S, \epsilon_i\in\left\{-1,1\right\}\right\}\,.$$
\item For $g\in G$, $g\neq 1$ the \emph{syllable length of $g$} is defined as
$$l_{sy}(g):=min\left\{m\mid g=s_1^{a_1}s_2^{a_2}\ldots s_m^{a_m}, s_i\in S, a_i\in\mathbb{Z}\right\}\,.$$
And we set $l_{sy}(1)=l(1)=0$.
\end{enumerate}
\end{definition}

We note that for a given group $G$ and a finite generating set $S$ consisting of elements of order two we have $l(g)=l_{sy}(g)$ for all $g\in G$.

The length of $g\in G$ is closely connected to the length of a special path in a geometric object  which is associated to the group $G$, the Cayley-graph $Cay(G,S)$. Before we give a definition of this graph we recall the definition and some important facts about general graphs which we will need later on.

A graph $\Gamma$ is a pair $(V(\Gamma),E(\Gamma))$ where $V(\Gamma)$ is a set whose elements are called \emph{vertices} and $E(\Gamma)$ is a subset of $\mathcal{P}_2(V):=\left\{X\mid X\subseteq V, |X|=2\right\}$ whose elements are called \emph{edges}. Usually, graphs are visualized graphically, where we draw for each vertex $x\in V(\Gamma)$ a point and label it with $x$ and two points $x, y$ are connected by a line if $\left\{x, y\right\}\in E(\Gamma)$. For example, the visualization of $\Gamma=(\left\{x_1, x_2, x_3\right\}, \left\{\left\{x_1,x_2\right\}\right\})$ is as shown in Figure 1.

\begin{figure}[h]
	\begin{center}
	\captionsetup{justification=centering}
		\begin{tikzpicture}
			\draw[fill=black]  (0,0) circle (2pt);
			\draw[fill=black]  (2,0) circle (2pt);
            \draw[fill=black]  (4,0) circle (2pt);
            \draw (0,0)--(2,0);
            \node at (0,-0.3) {$x_1$};
            \node at (2,-0.3) {$x_2$};
            \node at (4,-0.3) {$x_3$};
        \end{tikzpicture}
        \caption{The visualization of $\Gamma=(\left\{x_1, x_2, x_3\right\}, \left\{\left\{x_1,x_2\right\}\right\})$.}
    \end{center}
\end{figure} 

Given a graph $\Gamma$ and a vertex $x\in V(\Gamma)$ we define two subsets of $V(\Gamma)$ that are associated to $x$. The \emph{link of $x$}, denoted by $lk(x)$ is defined as $lk(x):=\left\{y\in V(\Gamma)\mid \left\{x,y\right\}\in E(\Gamma)\right\}$ and the \emph{star of $x$}, denoted by $st(x)$ is defined as $st(x):=lk(x)\cup\left\{x\right\}$. A graph $\Gamma$ is called \emph{complete} if $st(x)=V(\Gamma)$ for all $x\in V(\Gamma)$. A subgraph 
$\Omega\subseteq\Gamma$ is called \emph{full} if for all pair of vertices $(v,w)\in V(\Omega)\times V(\Omega)$ we have $\left\{v,w\right\}\in E(\Omega)$ if and only if $\left\{v,w\right\}\in E(\Gamma)$.

Let $G$ be a group and let $S$ be a generating set for $G$. The \emph{Cayley-graph for $G$ with respect to $S$}, denoted by $Cay(G,S)$ is a graph with vertex set $V(Cay(G,S))=G$ and edge set $E(Cay(G,S))=\left\{\left\{g, gs\right\}\mid g\in G, s\in S\cup S^{-1}\right\}$. The distance between two vertices is defined as a number of edges in a shortest path. Note that $l(g)$ is equal to the distance between the vertices $1_G$ and $g$. Hence the number of elements in $G$ with word length $n$ is equal to the number of vertices in the sphere with center $1_G$ of radius $n$ in $Cay(G,S)$.  Let us consider the Cayley-graph of the free group $F_2$ with the generating set $\left\{x,y\right\}$ in Figure 2.
$ $\\
 
\begin{figure}[h]
\begin{center}
\captionsetup{justification=centering}
\begin{tikzpicture}
\draw (0,-3)--(0,3);
\draw (-3,0)--(3,0);
\draw (-2,-1)--(-2,1);
\draw (-1,-2)--(1,-2);
\draw (2,-1)--(2,1);
\draw (-1,2)--(1,2);
\draw[fill=black]  (0,0) circle (1pt);
\node at (0.2,0.2) {\tiny{$1$}};
\draw[fill=black]  (2,0) circle (1pt);
\node at (1.8,0.2) {\tiny{$x$}};
\draw[fill=black]  (0,2) circle (1pt);
\node at (0.2,2.2) {\tiny{$y$}};
\draw[fill=black]  (-2,0) circle (1pt);
\node at (-1.7,0.2) {\tiny{$x^{-1}$}};
\draw[fill=black]  (0,-2) circle (1pt);
\node at (0.3,-1.8) {\tiny{$y^{-1}$}};
\draw[fill=black]  (2,1) circle (1pt);
\node at (2.3,1.2) {\tiny{$xy$}};
\draw[fill=black]  (2,-1) circle (1pt);
\node at (2.4,-0.8) {\tiny{$xy^{-1}$}};
\draw[fill=black]  (3,0) circle (1pt);
\node at (3.2,0.2) {\tiny{$xx$}};
\draw[fill=black]  (0,-3) circle (1pt);
\node at (0.55,-2.8) {\tiny{$y^{-1}y^{-1}$}};
\draw[fill=black]  (-1,-2) circle (1pt);
\node at (-1.5,-1.8) {\tiny{$y^{-1}x^{-1}$}};
\draw[fill=black]  (1,-2) circle (1pt);
\node at (1.4,-1.8) {\tiny{$y^{-1}x$}};
\draw[fill=black]  (-3,0) circle (1pt);
\node at (-3.4,0.2) {\tiny{$x^{-1}x$}};
\draw[fill=black]  (-2,1) circle (1pt);
\node at (-1.6,1.2) {\tiny{$x^{-1}y$}};
\draw[fill=black]  (-2,-1) circle (1pt);
\node at (-1.5,-0.8) {\tiny{$x^{-1}y^{-1}$}};
\draw[fill=black]  (0,3) circle (1pt);
\node at (0.2,3.2) {\tiny{$yy$}};
\draw[fill=black]  (-1,2) circle (1pt);
\node at (-1.3,2.2) {\tiny{$yx^{-1}$}};
\draw[fill=black]  (1,2) circle (1pt);
\node at (1.3,2.2) {\tiny{$yx$}};
\draw[dashed] (3,0)--(3.5,0);
\draw[dashed] (3,0.5)--(3,-0.5);
\draw[dashed] (-3,0)--(-3.5,0);
\draw[dashed] (-3,0.5)--(-3,-0.5);
\draw[dashed] (2,1)--(2,1.5);
\draw[dashed] (1.5,1)--(2.5,1);
\draw[dashed] (-2,1)--(-2,1.5);
\draw[dashed] (-1.5,1)--(-2.5,1);
\draw[dashed] (2,-1)--(2,-1.5);
\draw[dashed] (1.5,-1)--(2.5,-1);
\draw[dashed] (-2,-1)--(-2,-1.5);
\draw[dashed] (-1.5,-1)--(-2.5,-1);
\draw[dashed] (0,3)--(0,3.5);
\draw[dashed] (-0.5,3)--(0.5,3);
\draw[dashed] (0,-3)--(0,-3.5);
\draw[dashed] (-0.5,-3)--(0.5,-3);
\draw[dashed] (-1,2)--(-1.5,2);
\draw[dashed] (-1,1.5)--(-1,2.5);
\draw[dashed] (-1,-2)--(-1.5,-2);
\draw[dashed] (-1,-1.5)--(-1,-2.5);
\draw[dashed] (1,2)--(1.5,2);
\draw[dashed] (1,1.5)--(1,2.5);
\draw[dashed] (1,-2)--(1.5,-2);
\draw[dashed] (1,-1.5)--(1,-2.5);
\end{tikzpicture}
\caption{$Cay(F_2,\left\{x,y\right\}).$}
    \end{center}
\end{figure}
For example, the number of elements in $F_2$ with length $2$ is equal to $12$. One geometric way to count the elements in $F_2$ is by counting the vertices in the sphere with center $1$ of radius $n$ in the Cayley-graph $Cay(F_2,\left\{x,y\right\})$. Let $a_n$ be this number. We get the sequence $(a_n)_{n\in\mathbb{N}}$ where $a_n=4\cdot 3^{n-1}$. We convert this sequence into a formal power series $1+a_1t+a_2t^2+\ldots $ which leads us to 
the definition of the growth series. The best general reference on growth series of groups is \cite{Mann}.

\begin{definition}
Let $G$ be a group and $S$ be a finite generating set of $G$.
\begin{enumerate}
\item The \emph{growth series of $G$ with respect to $S$} is the formal series
$$\mathcal{G}_{(G,S)}(t):=\sum_{g\in G}t^{l(g)}=\sum_{n=0}^\infty |\left\{g\in G\mid l(g)=n\right\}|\cdot t^n\,.$$
\item The \emph{growth series of a subset $A\subseteq G$ with respect to $S$} is the formal series
$$\mathcal{G}_{(A,S)}(t):=\sum_{g\in A}t^{l(g)}=\sum_{n=0}^\infty |\left\{g\in A\mid l(g)=n\right\}|\cdot t^n\,.$$
\end{enumerate}
\end{definition}

We note that $\mathcal{G}_{(G,S)}(t)$ is an element in $\mathbb{Z}[[t]]$ the ring of formal power series in the variable $t$ over $\mathbb{Z}$. We now give some examples.

\begin{examples}
$ $
\begin{enumerate}
\item $\mathcal{G}_{(\mathbb{Z}/4\mathbb{Z}, \left\{1\right\})}(t)=1+2t+t^2$.
\item $\mathcal{G}_{(\mathbb{Z}/5\mathbb{Z},\left\{1\right\})}(t)=1+2t+2t^2$.
\item $\mathcal{G}_{(\mathbb{Z},\left\{1\right\})}(t)=1+2t+2t^2+2t^3+\ldots= 1+2(t+t^2+t^3\ldots)=1+\frac{2t}{1-t}=\frac{1+t}{1-t}$.
\end{enumerate}
\end{examples}

Given two groups $G$ and $H$ with finite generating sets $S_G$ resp. $S_H$, to construct a new group using  given ones it is natural to use direct or free product construction. For direct and free products, there are formulas for the growth series in terms of growth series of the factors \cite{Mann}. 
\begin{gather*}
\mathcal{G}_{(G\times H,S_G\cup S_H)}(t)=\mathcal{G}_{(G,S_G)}(t)\cdot \mathcal{G}_{(H,S_H)}(t)\,,\\
\frac{1}{\mathcal{G}_{(G*H, S_G\cup S_H)}(t)}=\frac{1}{\mathcal{G}_{(G, S_G)}(t)}+ \frac{1}{\mathcal{G}_{(H, S_H)}(t)}-1\,.
\end{gather*}

A generalization of direct resp. free product construction are amalgamated products and graph products of groups. Let us recall a formula for the growth series of an amalgamated product. First, we need a definition. 

\begin{definition}
Let $(G,S)$ be a pair where $G$ is a group generated by a finite set $S$. A pair $(H,T)$ is \emph{admissible} in $(G,S)$, if $H$ is a subgroup of $G$, $T\subseteq S$, and there exists a tranversal $U$ for $H$ in $G$ such that if $g=hu$ with $g\in G$, $h\in H$, $u\in U$, then $l_S(g)=l_T(h)+l_S(u)$. We always assume that the transversal contains the identity as the representative of $H$. 
\end{definition}

It was proven in \cite{Lewin} that if $(L,R)$ is admissible in $(H,S)$ and in $(K,T)$ then the growth series of $G=H*_L K$ can be computed using smaller pieces of $G$.

\begin{proposition}
\label{admissible}
If $(L,R)$ is admissible in $(H,S)$ and in $(K,T)$, then 
$$\frac{1}{\mathcal{G}_{(H*_L K, S\cup T)}(t)}=\frac{1}{\mathcal{G}_{(H, S)}(t)}+\frac{1}{\mathcal{G}_{(K, T)}(t)}-\frac{1}{\mathcal{G}_{(L, R)}(t)}\,.$$
\end{proposition}

Now we move on to graph products of groups. Given a finite graph $\Gamma$ and a collection of groups $G_x$ for $x\in V(\Gamma)$, the \emph{graph product of groups} is defined as
$$G_\Gamma=\big(\ast_{x\in V(\Gamma)} G_x\big)/\langle\langle [g,h]\mid g\in G_x, h\in G_y, \left\{x,y\right\}\in E(\Gamma)\rangle\rangle\,.$$
We note that, if $\Gamma$ is discrete, then the associated graph product of groups is the free product of the vertex groups and if $\Gamma$ is complete, then the associated graph product of groups is the direct product of the vertex groups. If all vertex groups are infinite cyclic, then we call $G_\Gamma$ a right-angled Artin group.
 For every vertex group $G_x$ let $S_x$ be a finite generating set and we set $S:=\bigcup_{x\in V(\Gamma)} S_x$. 
A formula for the growth series of $G_\Gamma$ in terms of the growth series of the vertex groups was proven in \cite{Lewin} for isomorphic vertex groups. Here we recall a special case of this  formula where $G_\Gamma$ is a right-angled Artin group and $S_x=\left\{1\right\}$ for every $x\in V$. Let $c_i$ be the number of complete subgraphs in $\Gamma$ on $i$ vertices. Then
$$\frac{1}{\mathcal{G}_{(G_\Gamma,S)}(t)}=\sum_i (-1)^i c_i \frac{(\frac{1+t}{1-t}-1)^i}{(\frac{1+t}{1-t})^i}\,.$$

This formula was generalized for arbitrary vertex groups in \cite{Chiwell}. Let $G_\Gamma$ be a graph product of finitely generated vertex groups. We define for each complete subgraph $\Delta\subseteq\Gamma$, $P_\Delta(t):=\prod_{x\in V(\Delta)}(\frac{1}{\mathcal{G}_{(G_x, S_x)}(t)}-1)$. Then
$$\frac{1}{\mathcal{G}_{(G_\Gamma, S)}(t)}=\sum P_\Delta(t)\,,$$
where the summation is taken over all complete subgraphs of $\Gamma$ including the empty one for which $P_\emptyset=1$. 

\bigskip
Further groups for which it is possible to compute the growth series using smaller building blocks of the group are Coxeter groups.
Coxeter groups have special subgroups which can be considered as building blocks for the whole group. Given a finite graph $\Gamma$ with an edge-labeling $m\colon E(\Gamma)\to\mathbb{N}_{\geq 2}$. The Coxeter group associated to $\Gamma$ is given by the presentation
$$W=\langle V(\Gamma)\mid x^2\text{ for all }x\in V(\Gamma), (xy)^{m(\left\{x,y\right\})}\text{ for all } \left\{x,y\right\}\in E(\Gamma)\rangle\,.$$

For any subset $X\subseteq V(\Gamma)$ the subgroup generated by the set $X$ is canonically isomorphic to the Coxeter group which is associated to the full subgraph of $\Gamma$ with the vertex set $X$. This subgroup is called a \emph{standard parabolic subgroup} and we denote it by $W_X$. A natural question is if it is possible to use the growth series of special parabolic subgroups to obtain a formula for the growth series of the whole group. It was proven in \cite{Solomon}, \cite{Bourbaki} that it is indeed the case.
Let $(W,S)$ be a Coxeter system.
If $W$ is finite, then
$$\frac{t^m+(-1)^{|S|+1}}{\mathcal{G}_{(W,S)}(t)}=\sum_{X\subsetneq S}\frac{(-1)^{|X|}}{\mathcal{G}_{(W_X,X)}(t)}\,,$$
where $m=max\left\{l(w)\mid w\in W\right\}$. The growth series of a finite Coxeter group can also be calculated using the non-recursive formula
$$\mathcal{G}_{(W,S)}(t)=\prod_{i=1}^k(1+t+\ldots+t^{m_i})\,,$$
where $m_1,\dots,m_k$ are the exponents of $(W,S)$.

If $W$ is infinite, then
$$\frac{(-1)^{|S|+1}}{\mathcal{G}_{(W,S)}(t)}=\sum_{X\subsetneq S}\frac{(-1)^{|X|}}{\mathcal{G}_{(W_X,X)}(t)}\,.$$

In particular, the above formulas show that for a Coxeter group $W$ there exists a polynomial $f(t)$ such that
$$\frac{f(t)}{\mathcal{G}_{(W,S)}(t)}=\sum_{X\subsetneq S}\frac{(-1)^{|X|}}{\mathcal{G}_{(W_X,X)}(t)}\,.$$

\section{Dyer groups}
We begin this chapter with the definition of the main protagonist in this article, a Dyer group. For two letters $a, b$ and a natural number $m$ we define $\pi(a,b,m):= abababa\ldots$ where the length of the word is $m$. For example $\pi(a,b,3)=aba$.

\begin{definition}
$ $
\begin{enumerate}
\item A \emph{Dyer graph} is a triple $(\Gamma, m, f)$ where $\Gamma$ is a graph with finite vertex set $V=V(\Gamma)$, $f\colon V\to\mathbb{N}_{\geq 2}\cup\left\{\infty\right\}$ and $m\colon E(\Gamma)\to\mathbb{N}_{\geq 2}$ are maps. For every edge $e=\left\{x,y\right\}\in E(\Gamma)$, if $m(e)\neq 2$, then $f(x)=f(y)=2$. 
\item The associated \emph{Dyer group} is defined as follows
\begin{gather*}
D:=\langle V\mid x^{f(x)}, x\in V \text{ if }f(x)\neq\infty, \pi(x,y,m(\left\{x,y\right\}))=\pi(y,x,m(\left\{x,y\right\}))\\
\text{ if }\left\{x,y\right\}\in E(\Gamma)\rangle\,.
\end{gather*}
\item The associated pair $(D,V)$ where $D$ is a Dyer group and $V=V(\Gamma)$ is called a \emph{Dyer system}.
\end{enumerate}
\end{definition}

\subsection{Dyer tools}
We start by recalling several results which were proven by Dyer in \cite{Dyer}. Let $G$ be a group and $g\in G$. We denote the order of $g$ by $o(g)$. If $o(g)$ is finite, then we write $\mathbb{Z}_{o(g)}$ for the cyclic group of cardinality $o(g)$ and if $o(g)$ is infinite, then we write $\mathbb{Z}_{o(g)}$ for the infinite cyclic group.
More generally we use the notation $\mathbb{Z}_n=\mathbb{Z}/n\mathbb{Z}$ if $n$ is a positive integer and $\mathbb{Z}_\infty=\mathbb{Z}$.

Let $(D,V)$ be a Dyer system. By definition, a conjugate of a generator $x\in V$ is called a \emph{reflection}. We define
$$R:=\left\{gxg^{-1}\mid g\in D, x\in V\right\}\,.$$
$R$ is the set of all reflections in $D$. For $\rho\in R$ we define a copy of $\mathbb{Z}_{o(\rho)}$ as $H_\rho=\left\{a[\rho]\mid a\in\mathbb{Z}_{o(\rho)}\right\}$. The set $H_\rho$ is an abelian group whose group operation is defined by $a[\rho]+b[\rho]:=(a+b)[\rho]$. Hence $H_\rho$ is isomorphic to $\mathbb{Z}_{o(\rho)}$. Further, we define
$$M=\bigoplus\limits_{\rho\in R}H_\rho\,.$$
This set is an abelian group with canonical group operation
$\sum a_\rho[\rho]+\sum b_\rho[\rho]=\sum (a_\rho+b_\rho)[\rho]$.
Furthermore, this abelian group is a $D$-module where the structure of the $D$-module is defined for $g\in D$ by
$$g\cdot \sum a_\rho[\rho]:=\sum a_\rho[g\rho g^{-1}]\,.$$

Let $g\in D$. We pick one syllabic representative $(x_1^{a_1}, x_2^{a_2}, \ldots, x_l^{a_l})$ for $g$, that is, a tuple of syllables such that $g=x_1^{a_1}x_2^{a_2}\ldots x_l^{a_l}$. For each $i\in\left\{1,\ldots, l\right\}$ we define a reflection
$$\rho_i:=x_1^{a_1}x_2^{a_2}\ldots x^{a_{i-1}}_{i-1}x_ix_{i-1}^{-a_{i-1}}\ldots x_2^{-a_2}x_1^{-a_1}\,.$$
We set
$$N(g)=\sum_{i=1}^l a_i[\rho_i]\in M\,.$$

For $n\in\mathbb{N}_{\ge 2}\cup\{\infty\}$ and $a\in\mathbb{Z}_n$ we denote by $\lVert a\rVert_n$ the word length of $a$ with respect to the generating set $\left\{1\right\}$.

\begin{theorem}(\cite{Dyer})
\label{DyerToolbar}
Let $(D,V)$ be a Dyer system. Let $g,h\in D$.
\begin{enumerate}
\item $N(g)$ does not depend on the choice of the syllabic representative for $g$.

\item Let $N(g)=\sum_{\rho\in R}a_\rho(g)[\rho]$. Then 
\begin{enumerate}
\item $l_{sy}(g)=\mid\left\{\rho\in R\mid a_\rho(g)\neq 0\right\}\mid$.
\item $l(g)=\sum_{\rho\in R}\lVert a_\rho(g)\rVert_{o(\rho)}$.
\end{enumerate}
\item $N(gh)=N(g)+g\cdot N(h)$.
\end{enumerate}
\end{theorem}

Let $g\in D$.
A syllabic representative $(x_1^{a_1},x_2^{a_2},\dots,x_l^{a_l})$ for $g$ is called \emph{reduced} if $l=l_{sy}(g)$.
The following is a direct consequence of part (2) of Theorem \ref{DyerToolbar} and it will be often used hereafter.

\begin{corollary}
\label{DyerConsequence}
Let $g\in D$ and $(x_1^{a_1},x_2^{a_2},\dots,x_l^{a_l})$ be a reduced syllabic representative for $g$.
Then
$$l(g)=\|a_1\|_{o(x_1)}+\|a_2\|_{o(x_2)}+\ldots+\|a_l\|_{o(x_l)}\,. $$
\end{corollary}

\begin{proof}
For each $i\in\left\{1,\ldots,l\right\}$ we set
$$\rho_i=x_1^{a_1}x_2^{a_2}\ldots x_{i-1}^{a_{i-1}}x_ix_{i-1}^{-a_{i-1}}\ldots x_2^{-a_2}x_1^{-a_1}\,.$$
Then
$$N(g)=\sum_{i=1}^la_i[\rho_i]\,.$$
Since $l=l_{sy}(g)$, by Theorem \ref{DyerToolbar}\,(2--a) we have $\rho_i\neq\rho_j$ for $i\neq j$.
By Theorem \ref{DyerToolbar}\,(2--b) it follows that
$$l(g)=\|a_1\|_{o(\rho_1)}+\|a_2\|_{o(\rho_2)}+\ldots+\|a_l\|_{o(\rho_l)} =\|a_1\|_{o(x_1)}+\|a_2\|_{o(x_2)}+\ldots+\|a_l\|_{o(x_l)}\,.$$
\end{proof}

\subsection{Standard parabolic subgroups}
Let $(D,V)$ be a Dyer system. 
For any subset $X\subseteq V$, we denote the subgroup generated by the set $X$ by $D_X\subseteq D$. 
$D_X$ is called the \emph{standard parabolic subgroup} generated by $X$. 

Let $(\Gamma,m,f)$ be the Dyer graph associated with $(D,V)$.
We denote by $\Gamma_X$ the full subgraph of $\Gamma$ spanned by $X$, by $m_X$ the restriction of $m$ to $E(\Gamma_X)$, and by $f_X$ the restriction of $f$ to $V(\Gamma_X)=X$.
Then $(\Gamma_X,m_X,f_X)$ is a Dyer graph and we know by \cite{Dyer} that $(D_X,X)$ is the Dyer system associated with $(\Gamma_X,m_X,f_X)$ (see also \cite[Proposition 2.7]{Paris}).

\begin{lemma}
\label{lengthparabolic}
Let $(D,V)$ be a Dyer system. 
Let $D_X$ be a standard parabolic subgroup of $D$. 
Then for any $g\in D_X$ we have $l_X(g)=l_V(g)$.
\end{lemma}

\begin{proof}
Let $g\in D_X$.
Let $(x_1^{a_1},x_2^{a_2},\ldots,x_l^{a_l})$ be a reduced syllabic representative for $g$.
We know by \cite[Lemma 2.5]{Paris} that $x_1,x_2,\dots, x_l\in X$, hence, by Corollary \ref{DyerConsequence},
$$l_V(g)=\|a_1\|_{o(x_1)}+\|a_2\|_{o(x_2)}+\ldots+\|a_l\|_{o(x_l)}\le l_X (g)\,.$$
It is clear that we also have $l_X(g)\ge l_V(g)$, thus $l_X(g)=l_V(g)$.
\end{proof}

\begin{proposition}
\label{minimallength}
Let $(D,V)$ be a Dyer system and $D_X$ be a standard parabolic subgroup. 
Then for every $g\in G$
\begin{enumerate}
\item there exists a unique $g_0\in gD_X$ such that
$$l_{sy}(g_0h)=l_{sy}(g_0)+l_{sy}(h)\text{ and } 
l(g_0h)=l(g_0)+l(h)$$ for all $h\in D_X$.
\item there exists a unique $g'_0\in D_Xg$ such that
$$l_{sy}(hg'_0)=l_{sy}(h)+l_{sy}(g'_0)\text{ and } 
l(hg'_0)=l(h)+l(g'_0)$$ for all $h\in D_X$.
\end{enumerate}
\end{proposition}

\begin{proof}
The statements regarding syllabic length were proved in \cite[Proposition 2.8]{Paris}.
Hence we know that there exists a unique $g_0\in gD_X$ such that for all $h\in D_X$ we have
$$l_{sy}(g_0h)=l_{sy}(g_0)+l_{sy}(h)\,.$$
Let $h\in D_X$.
Let $(x_1^{a_1},\ldots,x_p^{a_p})$ be a reduced syllabic representative for $g_0$ and $(y_1^{b_1},\ldots,y_q^{b_q})$ be a reduced syllabic representative for $h$.
We know from the above that $(x_1^{a_1},\ldots,x_p^{a_p},y_1^{b_1},\ldots,y_q^{b_q})$ is a reduced syllabic representative for $g_0h$, hence, by Corollary \ref{DyerConsequence},
$$l(g_0h)=\|a_1\|_{o(x_1)}+\ldots+\|a_p\|_{o(x_p)}+\|b_1\|_{o(y_1)}+\ldots+\|b_q\|_{o(y_q)}=l(g_0)+l(h)\,.$$

The proof of part (2) is the same as for part (1).
\end{proof}

\begin{corollary}
\label{Dyeradmissable}
Let $(\Gamma,m,f)$ be a Dyer graph and $(D,V)$ be the associated Dyer system. 
Every pair $(D_X,X)$ where $D_X$ is a standard parabolic subgroup is admissable.
\end{corollary}

\begin{proof}
For a standard parabolic subgroup $D_X$ and an element $g\in D$ there exists a unique $g_0\in gD_X$ such that $l(g_0h)=l(g_0)+l(h)$ for all $h\in D_X$. 
We take these minimal elements as a transversal. 
Lemma \ref{lengthparabolic} and Proposition \ref{minimallength} show that this transversal is admissible.
\end{proof}

\begin{corollary}
\label{amalgamdecomposition}
Let $(\Gamma,m,f)$ be a Dyer graph. 
Let $v\in V(\Gamma)$. 
If $st(v)\neq V(\Gamma)$, then 
$$D=D_{V-\left\{v\right\}}*_{D_{lk(v)}}D_{st(v)}\,,$$
and
$$\frac{1}{\mathcal{G}_{(D, V)}(t)}=\frac{1}{\mathcal{G}_{(D_{V-\left\{v\right\}}, V-\left\{v\right\})}(t)}+\frac{1}{\mathcal{G}_{(D_{st(v)}, st(v))}(t)}-\frac{1}{\mathcal{G}_{(D_{lk(v)},  lk(v))}(t)}\,.$$
\end{corollary}

\begin{proof}
The proof of the equality $D=D_{V-\left\{v\right\}}*_{D_{lk(v)}}D_{st(v)}$ follows by analyzing the presentation of $D$ and the canonical presentation of the amalgam.
Hence by Corollary \ref{Dyeradmissable} and Proposition \ref{admissible} we get
$$\frac{1}{\mathcal{G}_{(D, V)}(t)}=\frac{1}{\mathcal{G}_{(D_{V-\left\{v\right\}}, V-\left\{v\right\})}(t)}+\frac{1}{\mathcal{G}_{(D_{st(v)}, st(v))}(t)}-\frac{1}{\mathcal{G}_{(D_{lk(v)},  lk(v))}(t)}\,.$$
\end{proof}

\section{$X$-minimality}

\begin{definition}
Let $(D,V)$ be a Dyer system and $g\in D$. For $X\subseteq V$ the element $g$ is called \emph{$X$-minimal} if $l_{sy}(g)\leq l_{sy}(gh)$ fol all $h\in D_X$. 
\end{definition}

Note that, by Proposition \ref{minimallength}, if $g$ is $X$-minimal, then $l_{sy}(gh)=l_{sy}(g)+l_{sy}(h)$ and $l(gh)=l(g)+l(h)$ for all $h\in D_X$.
Note also that, if $X\subseteq Y\subseteq V$ and $g$ is $Y$-minimal, then $g$ is also $X$-minimal, since $gD_X\subseteq gD_Y$.

\begin{definition}
Let $(D,V)$ be a Dyer system and $X\subseteq V$. We define two subsets of $D$ as follows:
$$A_X=A_X(D):=\left\{g\in D\mid g\text{ is }X\text{-minimal}\right\} \text{ and } 
B_X=B_X(D):=A_X-(\cup_{X\subsetneq Y}A_Y)\,.$$
\end{definition}

The following lemma is a particular case of the well-known general M\"obius inversion formula (see \cite[Section 3.7]{Stanley} or \cite{Rota} for example).

\begin{lemma}
\label{Mobius}
Let $V$ be a set and $\mathcal{P}(V)$ be the set of all subsets of $V$. Further, let $G$ be an abelian group. If the functions $f\colon\mathcal{P}(V)\to G$ and $g\colon\mathcal{P}(V)\to G$ satisfy 
$$f(X)=\sum_{X\subseteq Y}g(Y)\text{ for all }X\in\mathcal{P}(V)\,,$$
then they satisfy
$$g(X)=\sum_{X\subseteq Y}(-1)^{|Y-X|}f(Y)\text{ for all }X\in\mathcal{P}(V)\,.$$
\end{lemma}

\begin{proposition}
\label{Firstformula}
Let $(D,V)$ be a Dyer system. For $X\subseteq V$ we have
$$\mathcal{G}_{(B_X, V)}(t)=\sum_{X\subseteq Y} (-1)^{|Y-X|}\frac{\mathcal{G}_{(D,V)}(t)}{\mathcal{G}_{(D_Y,Y)}(t)}\,.$$
In particular, for $X=\emptyset$ we obtain
$$\mathcal{G}_{(B_\emptyset, V)}(t)=\sum_{Y} (-1)^{|Y|}\frac{\mathcal{G}_{(D,V)}(t)}{\mathcal{G}_{(D_Y,Y)}(t)}\,,$$
which is equivalent to
$$\frac{\mathcal{G}_{(B_\emptyset, V)}(t)+(-1)^{|V|+1}}{\mathcal{G}_{(D,V)}(t)}=\sum_{Y\subsetneq V}\frac{(-1)^{|Y|}}{\mathcal{G}_{(D_Y,Y)}(t)}\,.$$
\end{proposition}

\begin{proof}
Let $(D,V)$ be a Dyer system. We define two functions $f\colon\mathcal{P}(V)\to\mathbb{Z}[[t]]$ and $g\colon\mathcal{P}(V)\to\mathbb{Z}[[t]]$ where $\mathbb{Z}[[t]]$ is the formal power series ring with coefficients in the group $\mathbb{Z}$ as follows
$$f(X)(t)=\sum_{g\in A_X}t^{l(g)}\text{ and } g(X)(t)=\sum_{g\in B_X}t^{l(g)}\,.$$
Note that $A_X$ is a disjoint union of those $B_Y$, where $X\subseteq Y$. Hence we have
$$f(X)(t)=\sum_{X\subseteq Y}g(Y)(t)\,,$$
and by Lemma \ref{Mobius} we obtain
$$g(X)(t)=\sum_{X\subseteq Y}(-1)^{|Y-X|}f(Y)(t)\,.$$
By definition we have $g(X)(t)=\mathcal{G}_{(B_X,V)}(t)$. Thus we obtain
$$\mathcal{G}_{(B_X,V)}(t)=\sum_{X\subseteq Y}(-1)^{|Y-X|}f(Y)(t)\,.$$
Further, $D=\cup_{g\in A_Y}gD_Y$ and this union is disjoint, hence
$$\mathcal{G}_{(D,V)}(t)=\sum_{g\in A_Y}\sum_{u\in D_Y}t^{l(g)+l(u)}=f(Y)(t)\cdot \mathcal{G}_{(D_Y,V)}(t)\,.$$
Finally, by Lemma \ref{lengthparabolic} we have $\mathcal{G}_{(D_Y,V)}(t)=\mathcal{G}_{(D_Y,Y)}(t)$, hence
$$\mathcal{G}_{(B_X,V)}(t)=\sum_{X\subseteq Y}(-1)^{|Y-X|}f(Y)(t)=\sum_{X\subseteq Y}(-1)^{|Y-X|}\frac{\mathcal{G}_{(D,V)}(t)}{\mathcal{G}_{(D_Y,Y)}(t)}\,.$$
\end{proof}

Our next task is to give a good description of the set $B_\emptyset$. 
We are particularly interested in properties of $(D,V)$ that ensure the set $B_\emptyset$ to be empty.

\begin{definition}
Let $(\Gamma,m,f)$ be a Dyer graph and $(D,V)$ be the associated Dyer system. 
We define $V_2:=\left\{x\in V| f(x)=2\right\}$, $V_\infty:=\left\{x\in V| f(x)=\infty\right\}$, $V_p:=\left\{x\in V| 2<f(x)<\infty \right\}$.
Let $D_2$ resp. $D_\infty$ resp. $D_p$ be the standard parabolic subgroup of $D$ generated by $V_2$ resp. $V_\infty$ resp. $V_p$.

The Dyer group $D$ is called of \emph{spherical type} if $\Gamma$ is a complete graph and $D_2$ is a finite Coxeter group.
\end{definition}

Note that, if $D$ is of spherical type, then $D=D_2\times D_p\times D_\infty$, $D_p=\prod_{x\in V_p} \mathbb{Z}/f(x)\mathbb{Z}$, and $D_\infty=\mathbb{Z}^l$ where $l=|V_\infty|$.

\bigskip
The description of $B_\emptyset$ when $D=D_2$ is a Coxeter group is well-known and it is a direct consequence of the following.

\begin{proposition}(\cite{Bourbaki})
\label{MaxCoxeter}
Let $(W,S)$ be a Coxeter system.
The following conditions on an element $w_0\in W$ are equivalent.
\begin{itemize}
\item[(a)]
For each $u\in W$, $l(w_0)=l(w_0u^{-1})+l(u)$.
\item[(b)]
For each $s\in S$, $l(w_0s)<l(w_0)$.
\end{itemize}
Moreover, $w_0$ exists if and only if $W$ is finite.
If $w_0$ satisfies (a) and/or (b), then $w_0$ is unique, $w_0$ is an involution, and $w_0Sw_0=S$.
\end{proposition}

The element $w_0$ of Proposition \ref{MaxCoxeter} is called the \emph{longest element} of $W$, if it exists.
The following is a straightforward consequence of Proposition \ref{MaxCoxeter}.

\begin{corollary}
Let $(W,S)$ be a Coxeter system.
We have $B_\emptyset(W)\neq\emptyset$ if and only if $W$ is finite.
If $W$ is finite, then $B_\emptyset(W)=\{w_0\}$, where $w_0$ is the longest element of $W$.
\end{corollary}

In the general case we have the following.

\begin{lemma}
\label{MaxDyer}
Let $(D,V)$ be a Dyer system.
\begin{enumerate}
\item
We have $B_\emptyset\neq\emptyset$ if and only if $D$ is of spherical type.
\item
Suppose $D$ is of spherical type.
Set $V_p=\{x_1,\dots,x_k\}$ and $V_\infty=\{y_1,\dots,y_l\}$.
Let $g\in D$.
Then $g\in B_\emptyset$ if and only if $g$ can be written in the form
$$g=w_0x_1^{a_1}\ldots x_k^{a_k}y_1^{b_1}\ldots y_l^{b_l}\,,$$
where $w_0$ is the longest element of $D_2$, $a_i\in(\mathbb{Z}/f(x_i)\mathbb{Z})-\{0\}$ for all $i\in\{1,\ldots,k\}$, and $b_j\in\mathbb{Z}-\{0\}$ for all $j\in\{1,\ldots,l\}$.
\end{enumerate}
\end{lemma}

\begin{proof}
We first prove that, if $B_\emptyset\neq\emptyset$, then $D$ is of spherical type.
We will then show that, if $D$ is of spherical type, then $B_\emptyset\neq\emptyset$ and the elements of $B_\emptyset$ are as described in part (2).

Suppose $B_\emptyset\neq\emptyset$.
This means that there exists $g\in D$ such that $g\not\in A_{\{x\}}$ for all $x\in V$.
So, we can pick $g\in D$ such that, for all $x\in V$, there exists $a\in\mathbb{Z}_{f(x)}-\{0\}$ such that $l_{sy}(gx^a)\le l_{sy}(g)$.
We start by showing that $\Gamma$ is complete.
Let $x,y\in V$, $x\neq y$.
Set $X=\{x,y\}$.
We know that there exist $a\in\mathbb{Z}_{f(x)}-\{0\}$ and $b\in\mathbb{Z}_{f(y)}-\{0 \}$ such as $l_{sy}(gx^a)\le l_{sy}(g)$ and $l_{sy}(gy^b)\le l_{sy}(g)$.
On the other hand there exist $g_0\in A_X$ and $h\in D_X$ such that $g=g_0h$.
By Proposition \ref{minimallength} we have
$$l_{sy}(g)=l_{sy}(g_0)+l_{sy}(h)\,,\ l_{sy}(gx^a)=l_{sy}(g_0)+l_{sy }(hx^a)\,,\ l_{sy}(gy^b)=l_{sy}(g_0)+
l_{sy}(hy^b)\,.$$
Thus, $l_{sy}(hx^a)\le l_{sy}(h)$ and $l_{sy}(hy^b)\le l_{sy}(h)$.
If $x$ and $y$ are not connected by an edge, then $D_X=\mathbb{Z}_{f(x)}*\mathbb{Z}_{f(y)}$ and there is no $h$ in $\mathbb{Z}_{f(x)}*\mathbb{Z}_{f(y)}$ such that $l_{sy}(hx^a)\le l_{ sy}(h)$ and $l_{sy}(hy^b)\le l_{sy}(h)$.
So, $x$ and $y$ are connected by an edge.

Since $\Gamma$ is a complete graph, we have $D=D_2\times D_p\times D_\infty$, $D_p=\prod_{x\in V_p}\mathbb{Z}/f(x)\mathbb{Z} $, and $D_\infty=\mathbb{Z}^l$, where $l=|V_\infty|$.
Let $g\in B_\emptyset$ that we write in the form $g=g_2g_pg_\infty$ with $g_2\in D_2$, $g_p\in D_p$, and $g_\infty\in D_\infty$.
For each $x\in V_2$ we have
$$l(g_2x)+l_{sy}(g_p)+l_{sy}(g_\infty)=l_{sy}(gx)\le l_{sy}(g)=l(g_2)+l_{sy}(g_p)+l_{sy}(g_\infty)\,,$$
hence $l(g_2x)\le l(g_2)$.
By Proposition \ref{MaxCoxeter} this implies that $D_2$ is a finite Coxeter group and $g_2$ is the longest element of $D_2$.
So, if $B_\emptyset\neq\emptyset$, then $D$ is of spherical type.

Suppose now that $D$ is of spherical type.
Then $D=D_2\times D_p\times D_\infty$, $D_2$ is a finite Coxeter group, $D_p=\prod_{x\in V_p}\mathbb{Z}/f(x)\mathbb{Z }$, and $D_\infty=\mathbb{Z}^l$, where $l=|V_\infty|$.
Set $V_p=\{x_1,\dots,x_k\}$ and $V_\infty=\{y_1,\dots,y_l\}$.
Let $g\in B_\emptyset$.
Write $g$ in the form $g=wx_1^{a_1}\ldots x_k^{a_k}y_1^{b_1}\ldots y_l^{b_l}$ with $w\in D_2$, $a_i\in\mathbb{Z}/f(x_i)\mathbb{Z} $ for all $i\in\{1,\dots,k\}$, and $b_j\in\mathbb{Z}$ for all $j\in\{1,\ldots,l\}$.
Let $i\in\{1,\ldots,k\}$.
If we had $a_i=0$, then we would have $l_{sy}(gx_i^c)>l_{sy}(g)$ for all $c\in(\mathbb{Z}/f(x_i)\mathbb {Z})-\{0\}$, hence we would have $g\not\in B_\emptyset$.
So $a_i\neq0$ for all $i\in\{1,\dots,k\}$.
Similarly, $b_j\neq0$ for all $j\in\{1,\ldots,l\}$.
If $w$ were not the longest element of $D_2$, then there would exist $x\in V_2$ such that $l(wx)>l(w)$, hence there would exist $x\in V_2$ such that $l_{sy}(gx)>l_{sy}(g)$.
So, $w$ is the longest element of $D_2$.

Let $g\in D$ which can be written in the form $g=w_0x_1^{a_1}\ldots x_k^{a_k}y_1^{b_1}\ldots y_l^{b_l}$, where $w_0$ is the longest element of $ D_2$, $a_i\in(\mathbb{Z}/f(x_i)\mathbb{Z})-\{0\}$ for all $i\in\{1,\ldots,k\}$, and $ b_j\in\mathbb{Z}-\{0\}$ for all $j\in\{1,\ldots,l\}$.
Notice that such an element always exists.
It is easily seen that $l_{sy}(gx_i^{-a_i})<l_{sy}(g)$ for all $i\in\{1,\ldots,k\}$ and $l_{sy}(gy_j^{-b_j})<l_{sy}(g)$ for all $j\in\{1,\ldots,l\}$.
On the other hand, if $x\in V_2$, then $l(w_0x)<l(w_0)$, hence $l_{sy}(gx)<l_{sy}(g)$.
So, $g\in B_\emptyset$.
\end{proof}

Let $(D,V)$ be a Dyer system of spherical type.
So, $D=D_2\times D_p\times D_\infty$, $D_2$ is a finite Coxeter group, $D_p=\prod_{x\in V_p}\mathbb{Z}/f(x)\mathbb {Z}$, and $D_\infty=\mathbb{Z}^l$, where $l=|V_\infty|$.
Let $x\in V_p$.
If $f(x)=2r$ is even we set $P_x(t)=2t+2t^2+\ldots+2t^{r-1}+t^r$, and if $f(x)=2r+1$ is odd we set $P_x(t)=2t+2t^2+\ldots+2t^r$.
Then we set 
$$P_D(t)=t^m\left(\prod_{x\in V_p}P_x(t)\right)\frac{2^lt^l}{(1-t)^l}\,,$$
where $m$ is the maximal length in $D_2$.

As an immediate corollary we obtain the following whose first part finishes the proof of Theorem \ref{MainTheorem}.

\begin{corollary}
\label{CorPart1}
Let $(D,V)$ be a Dyer system. 
\begin{enumerate}
\item 
If $D$ is not of spherical type, then
$$\frac{(-1)^{|V|+1}}{\mathcal{G}_{(D,V)}(t)}=\sum_{Y\subsetneq V}\frac{(-1)^{|Y|}}{\mathcal{G}_{(D_Y,Y)}(t)}\,.$$
\item 
If $D$ is of spherical type, then 
$$\frac{P_D(t)+(-1)^{|V|+1}}{\mathcal{G}_{(D,V)}(t)}=\sum_{Y\subsetneq V}\frac{(-1)^{|Y|}}{\mathcal{G}_{(D_Y,Y)}(t)}\,.$$
\end{enumerate}
\end{corollary}

\begin{proof}
By Proposition \ref{Firstformula} it suffices to show that $\mathcal{G}_{(B_\emptyset,V)}(t)=0$ if $D$ is not of spherical type and that $\mathcal {G}_{(B_\emptyset,V)}(t)=P_D(t)$ if $D$ is of spherical type.
If $D$ is not of spherical type then, by Lemma \ref{MaxDyer}, $B_\emptyset=\emptyset$, hence $\mathcal{G}_{(B_\emptyset,V)}(t)=0$.
Suppose $D$ is of spherical type.
Let $V_p=\{x_1,\dots,x_k\}$ and $V_\infty=\{y_1,\dots,y_l\}$, and let $m$ be the maximal length in $D_2$.
Then, by Lemma \ref{MaxDyer},
\begin{gather*}
\mathcal{G}_{(B_\emptyset,V)}(t)=t^m\left(\prod_{i=1}^k \mathcal{G}_{(\mathbb{Z}/f(x_i)\mathbb{Z}-\{0\},\{1\})}(t)
\right)\left(\prod_{i=1}^l\mathcal{G}_{(\mathbb{Z}-\{0\},\{1\})}(t)\right)=\\
t^m\left(\prod_{i=1}^kP_{x_i}(t)
\right)\left(\frac{2t}{1-t}\right)^l=P_D(t)\,.
\end{gather*}
\end{proof}

We end this chapter with the proof of Theorem \ref{MainTheorem}.
\begin{proof}[Proof of Theorem \ref{MainTheorem}]
The first part follows from Corollary \ref{CorPart1} which we already mentioned above. We assume now that $D$ is of spherical type, then $D=D_2\times D_p\times D_\infty$. Hence we can use the formula for direct products on page 5. We get
$$\mathcal{G}_{(D,V)}(t)=\mathcal{G}_{(D_2,V_2)}(t)\cdot \mathcal{G}_{(D_p,V_p)}(t)\cdot \mathcal{G}_{(D_\infty,V_\infty)}(t)\,.$$
Further, since $\mathcal{G}_{(\mathbb{Z},\left\{1\right\})}(t)=\frac{1+t}{1-t}$ we get 
$\mathcal{G}_{(D_\infty,V_\infty)}(t)=\frac{(1+t)^l}{(1-t)^l}$ where $l$ is the cardinality of $V_\infty$. A direct calculation shows the formulas for the growth series of finite cyclic groups with the standard generating sets which ends the proof of the second part of the theorem.
\end{proof}

\section{Euler characteristic}

We start this chapter by recalling the definition and useful formulas of the Euler characteristic of groups. Following \cite{Brown} a group $G$ is said to be of \emph{finite homological type} if the virtual cohomological dimension of $G$ is finite and for every $G$-module $M$ which is finitely generated as an abelian group, $H_i(G,M)$ is finitely generated for all $i$. If $G$ is torsion-free and of finite homological type, then its Euler characteristic is defined by
$$\chi(G):=\sum (-1)^i rk_{\mathbb{Z}}(H_i(G))\,.$$
If $G$ is of finite homological type and has a torsion free subgroup $H$ of finite index, then the Euler characteristic of $G$ is defined by
$$\chi(G):=\frac{\chi(H)}{[G:H]}\,.$$
We list some useful properties of the Euler characteristic.

\begin{proposition}(\cite[Proposition 7.3]{Brown})
\label{Brownformulas}
\begin{enumerate}
\item Let $1\rightarrow A\rightarrow B\rightarrow C\rightarrow 1$ be a short exact sequence where $A$ and $C$ are of finite homological type. If $B$ is virtually torsion-free, then $B$ is of finite homolocial type and
$$\chi(B)=\chi(A)\cdot\chi(C)\,.$$
\item Let $G=A*_B C$ be an amalgamated product where $A, B, C$ are of finite homological type. If $G$ is virtually torsion free, then $G$ is of finite homological type and
$$\chi(G)=\chi(A)+\chi(C)-\chi(B)\,.$$
\end{enumerate}
\end{proposition}

As a corollary we obtain

\begin{corollary}
\label{EulerDyerFormulas}
Let $(D,V)$ and $(D', V')$ be Dyer systems.
\begin{enumerate}
\item $\chi(D\times D')=\chi(D)\cdot\chi(D')$.
\item If $D=D_{V-\left\{x\right\}}*_{D_{lk(x)}} D_{st(x)}$, then
$$\chi(D)=\chi(D_{V-\left\{x\right\}})+\chi(D_{st(x)})-\chi(D_{lk(x)})\,.$$
\end{enumerate}
\end{corollary}

\begin{proof}
It was proven in \cite[Corollary 1.2]{Soergel} that every Dyer group is a subgroup of finite index in a Coxeter group.  
Further, it was proven in \cite{Serre} that Coxeter groups are of finite homological type. Since the property of being of finite homological type is preserved by taking finite index subgroups \cite[Lemma 6.1]{Brown}, we know that every Dyer group is of finite homological type and is therefore virtually torsion free. Proposition \ref{Brownformulas} shows the results of the corollary.
\end{proof}

\begin{theorem}
Let $(\Gamma, m,f)$ be a Dyer graph and $(D,V)$ be the associated Dyer system. Then
$$\frac{1}{\mathcal{G}_{(D,V)}(1)}=\chi(D).$$
\end{theorem}

\begin{proof}
Assume first that $\Gamma$ is complete. In this case $D=D_2\times D_p\times D_\infty$,
where $D_p$ is finite and $D_\infty\cong\mathbb{Z}^l$.
Hence, by Corollary \ref{EulerDyerFormulas}
$$\chi(D)=\chi(D_2)\cdot\chi(D_p)\cdot\chi(\mathbb{Z}^l)\,.$$
Since $D_2$ is a Coxeter group we know by \cite{Serre} that $\chi(D_2)=\frac{1}{\mathcal{G}_{(D_2, V_2)}(1)}$. We get
$$\chi(D)=\frac{1}{\mathcal{G}_{(D_2, V_2)}(1)}\cdot\chi(D_p)\cdot\chi(\mathbb{Z}^l)\,.$$
Further, since $D_p$ is finite we have $\chi(D_p)=\frac{1}{\mid D_p\mid}$ and $\mathcal{G}_{(D_p, V_p)}(1)=|D_p|$. Thus
$$\chi(D)=\frac{1}{\mathcal{G}_{(D_2, V_2)}(1)}\cdot\frac{1}{\mathcal{G}_{(D_p, V_p)}(1)}\cdot\chi(\mathbb{Z}^l)\,.$$
We know that $\chi(\mathbb{Z}^l)=0$ and $\frac{1}{\mathcal{G}_{(D_\infty, V_\infty)}(1)}=\frac{(1-1)^l}{(1+1)^l}=0$ if $l>0$, and $\chi(\mathbb{Z}^l)=1$ and $\frac{1}{\mathcal{G}_{(D_\infty, V_\infty)}(1)}=1$ if $l=0$. Hence
$$\chi(D)=\frac{1}{\mathcal{G}_{(D_2, V_2)}(1)}\cdot\frac{1}{\mathcal{G}_{(D_p, V_p)}(1)}\cdot\frac{1}{\mathcal{G}_{(D_\infty, V_\infty)}(1)}=\frac{1}{\mathcal{G}_{(D,V)}(1)}\,.$$

Now assume that $\Gamma$ is not complete, then 
there exists $x\in V(\Gamma)$ such that $st(x)\neq V(\Gamma)$. Then we have
$$D=D_{V-\left\{x\right\}}*_{D_{lk(x)}}D_{st(x)}\,.$$ By Corollary \ref{amalgamdecomposition} we obtain
$$\frac{1}{\mathcal{G}_{(D, V)}(1)}=\frac{1}{\mathcal{G}_{(D_{V-\{x\}}, V-\{x\})}(1)}+\frac{1}{\mathcal{G}_{(D_{st(x)}, st(x))}(1)}-\frac{1}{\mathcal{G}_{(D_{lk(x)}, lk(x))}(1)}\,.$$
By Corollary \ref{EulerDyerFormulas} we get
$$\chi(D)=\chi(D_{V-\left\{x\right\}}*_{D_{lk(x)}}D_{st(x)})=\chi(D_{V-\{x\}})+\chi(D_{st(x)})-\chi(D_{lk(x)})\,.$$
We decompose $D_{V-\left\{x\right\}}$, $D_{st(x)}$ and $D_{lk(x)}$ again in amalgamated products. Using this strategy we will get a linear combination of Euler characteristics resp. growth series of standard parabolic subgroups of $D$ where all defining graphs are complete. Applying the above formulas we get
$$\frac{1}{\mathcal{G}_{(D,V)}(1)}=\chi(D)\,.$$
\end{proof}

%%0%%

\end{document}